\documentclass{birkjour}
 \newtheorem{thm}{Theorem}[section]
 \newtheorem{cor}[thm]{Corollary}
 
 \newtheorem{prop}[thm]{Proposition}
 \theoremstyle{definition}
 
 \theoremstyle{remark}
 \newtheorem{rem}[thm]{Remark}
 
 \numberwithin{equation}{section}
 
\usepackage{color}
\usepackage{amstext}
\usepackage{amsthm}
\usepackage{amssymb}

\newcommand{\cC}{\mathcal{C}}

\newcommand{\cM}{\mathcal{M}}

\newcommand{\abs}[1]{\ensuremath{|#1|}}

\newcommand{\Abs}[1]{\ensuremath{\left|#1\right|}}

\newcommand{\Norm}[2]{\ensuremath{\|#1\|_{#2}}}

\renewcommand{\epsilon}{\varepsilon}
\renewcommand{\phi}{\varphi}

\begin{document}
\sloppy

\title[$H^{\infty}$ interpolation and embedding theorems for rational functions]{$H^{\infty}$ interpolation and embedding theorems \\
 for rational functions}

\author{Anton Baranov}

\address{Department of Mathematics and Mechanics, Saint Petersburg State University,
28, Universitetskii pr., St. Petersburg, 198504, Russia}

\email{anton.d.baranov@gmail.com}

\thanks{The work is supported by Russian Science Foundation grant 14-41-00010.}

\author{Rachid Zarouf}

\address{Aix-Marseille Universit\'{e}, Laboratoire Apprentissage, Didactique, Evaluation,
Formation, 32 Rue Eug\`{e}ne Cas CS 90279 13248 Marseille Cedex 04, France \\
\phantom{r} and \\
Department of Mathematics and Mechanics, Saint Petersburg State University,
28, Universitetskii pr., St. Petersburg, 198504, Russia}

\email{rachid.zarouf@univ-amu.fr}

\subjclass{Primary 15A60, 32A36, 26A33; Secondary 30D55,
26C15, 41A10}

\keywords{$H^{\infty}$ interpolation, Blaschke product, Model space, Rational function, Hardy spaces, Weighted Bergman spaces}

\date{February 25, 2019}
\begin{abstract}
We consider a Nevanlinna--Pick interpolation problem on finite sequences
of the unit disc $\mathbb{D}$ constrained by Hardy and radial-weighted
Bergman norms. We find sharp asymptotics on the corresponding interpolation
constants. As another application of our techniques we prove embedding
theorems for rational functions. We find that the embedding of $H^{\infty}$
into Hardy or radial-weighted Bergman spaces in $\mathbb{D}$ is invertible
on the subset of rational functions of a given degree $n$ whose poles
are separated from the unit circle and obtain asymptotically sharp
estimates of the corresponding embedding constants. 
\end{abstract}

\maketitle

\bigskip{}

\section{\label{sec:Statement-of-the}Introduction }

We denote by $\mathbb{D}=\{z\in\mathbb{C}:\,\vert z\vert<1\}$ the
unit disc and by $\mathcal{H}ol\left(\mathbb{D}\right)$ the space
of holomorphic functions in $\mathbb{D}$. We consider the following
Banach spaces $X\subset\mathcal{H}ol\left(\mathbb{D}\right)$: 
\begin{enumerate}
\item the Hardy spaces $X=H^{p}=H^{p}(\mathbb{D}),\:1\leq p\le\infty$;
we refer to \cite{Dur} for the corresponding definition and their
general properties; 
\item the radial-weighted Bergman spaces $X=A^{p}\big((1-\left|z\right|^{2})^{\beta}{\rm d}\mathcal{A}\big)=A^{p}\left(\beta\right)$, $1\leq p<\infty$, $\beta>-1$: 
\[
X=\Big\{ f\in\mathcal{H}ol(\mathbb{D})\,:\;\Norm{f}{A^{p}\left(\beta\right)}^{p}=\int_{\mathbb{D}}\left|f(z)\right|^{p}(1-\left|z\right|^{2})^{\beta}{\rm d}\mathcal{A}(z)<\infty\Big\},
\]
where $\mathcal{A}$ is the normalized area measure on $\mathbb{D}$.
We refer to \cite{HKZ} for general properties of $A^{p}\left(\beta\right)$.
For $\beta=0$ we shorten the notation to $X=A^{p}$. 
\end{enumerate}

\subsection{\label{sub:Problem-1.-Effective} Effective $H^{\infty}$ interpolation }

We consider the following problem: given a Banach space $X\subset\mathcal{H}ol\left(\mathbb{D}\right)$
and a finite sequence $\sigma$ in $\mathbb{D}$, what is the best
possible interpolation of the traces $f\vert_{\sigma}$, $f\in X$,
by functions from the space $H^{\infty}$? The case $X\subset H^{\infty}$
is of no interest (such a situation implies the uniform
boundedness of the interpolation quantity $c(\sigma,\,X,\,H^{\infty})$
below), and so one can suppose that either $H^{\infty}\subset X$
or $X$ and $H^{\infty}$ are incomparable. More precisely, our problem
is to compute or estimate the following interpolation quantity 
\[
c(\sigma,\,X,\,H^{\infty})={\displaystyle \sup_{f\in X,\,\Norm{f}{X}\leq1}}\mbox{inf}\big\{\Norm{g}{\infty}:\:g\in H^{\infty},\:g\vert_{\sigma}=f\vert_{\sigma}\big\}.
\]

It is discussed in \cite{Zar2} that the classical interpolation problems,
those of Nevanlinna--Pick and Carath\'eodory--Schur (see \cite[p. 231]{Nik1})
on one hand and Carleson's free interpolation (see \cite[p. 158]{Nik2})
on the other hand, are of this nature. For general Banach spaces $X$
containing $H^{\infty}$ as a dense subset,
$c\left(\sigma,\,X,\,H^{\infty}\right)$ is expressed as 
\[
c(\sigma,\,X,\,H^{\infty})={\displaystyle \sup_{f\in X\cap H^{\infty},\,\Norm{f}{X}\leq1}}\Norm{f}{H^{\infty}/B_{\sigma}H^{\infty}},
\]
where $B_{\sigma}$ is the finite Blaschke product 
\[
B_{\sigma}=\prod_{\lambda\in\sigma}b_{\lambda},\qquad b_{\lambda}=\frac{\lambda-z}{1-\overline{\lambda}z},
\]
$b_{\lambda}$ being the elementary Blaschke factor associated to
a $\lambda\in\mathbb{D}$. We denote by $\sigma_{n,\,\lambda}=(\lambda,\,...,\,\lambda)\in\mathbb{D}^{n}$
the one-point sequence of multiplicity $n$ corresponding to a given
$\lambda\in\mathbb{D}$.

It is a natural problem (related, e.g., to matrix analysis) to study
the asymptotic behaviour of $c\left(\sigma,\,X,\,H^{\infty}\right)$
when the set $\sigma$ approaches the boundary and its cardinality
tends to infinity. We put 
\[
C_{n,\,r}(X,\,H^{\infty})=\sup\,\big\{ c\left(\sigma,\,X,\,H^{\infty}\right):\:\sigma\in\mathbb{D}^{n},\:\:\max_{\lambda\in\sigma}\left|\lambda\right|\leq r\big\}.
\]
Initially motivated by a question posed in an applied context in \cite{Bar,BaWi},
asymptotically sharp estimates of $C_{n,\,r}(X,\,H^{\infty})$ were
derived in \cite{Zar2} for the cases $X=H^{p}$, $p\in2\mathbb{N}$,
and $X=A^{2}$. 
\begin{thm}
\textup{\cite{Zar2}} Given $n\geq1$, $r\in[0,\,1)$, $p\in2\mathbb{N}$
and $\lambda\in\mathbb{D}$ with $\left|\lambda\right|\leq r$, we
have 
\begin{equation}
a_{p}\left(\frac{n}{1-\left|\lambda\right|}\right)^{\frac{1}{p}}\leq c(\sigma_{n,\,\lambda},\,H^{p},H^{\infty})\leq C_{n,\,r}(H^{p},\,H^{\infty})\leq b_{p}
\left(\frac{n}{1-r}\right)^{\frac{1}{p}},\label{eq:Hardy_interp}
\end{equation}
\begin{equation}
a\cdot\frac{n}{1-\left|\lambda\right|}\leq c(\sigma_{n,\,\lambda},\,A^{2},\,H^{\infty})\leq C_{n,\,r}(A^{2},\,H^{\infty})\leq b\cdot\frac{n}{1-r},\label{eq:Bergm_interp}
\end{equation}
where $a_{p},b_{p}$ are constants depending only on $p$ and $a,b$
are some absolute constants. 
\end{thm}
\begin{rem}
The right-hand side inequality in \eqref{eq:Hardy_interp} is established
in \cite{Zar2} for any $p\in[1,\,\infty)$. The proof makes use of
a deep interpolation result between Hardy spaces by P. Jones \cite{Jon},
which we avoid in the present paper. 
\end{rem}
From now on, for two positive functions $a$ and $b$, we say that
$a$ is dominated by $b$, denoted by $a\lesssim b$, if there is
a constant $c>0$ such that $a\leq cb;$ and we say that $a$ and
$b$ are comparable, denoted by $a\asymp b$, if both $a\lesssim b$
and $b\lesssim a$.

The following conjecture for general Banach spaces $X$ (of analytic
functions of moderate growth in $\mathbb{D}$) was formulated in \cite{Zar2}:
\begin{equation}
C_{n,\,r}\left(X,\,H^{\infty}\right)\asymp\varphi_{X}\left(1-\frac{1-r}{n}\right),\label{eq:Conjecture}
\end{equation}
where $\varphi_{X}(t)$ stands for the norm of the evaluation functional
$f\mapsto f(t)$ on the space $X$. One of the main results of \cite{Zar3}
verifies the conjecture \eqref{eq:Conjecture} for the case $X=A^{2}(\beta)$,
$\beta\in\mathbb{Z}_{+}$. More recently an upper bound on $c(\sigma_{n,\,\lambda},\,A^{p}(\beta),\,H^{\infty})$
with $1\leq p\leq2$, and $\beta>-1$ was derived in \cite{Zar1}.
\medskip{}

In this paper we \smallskip{}

\begin{enumerate}
\item strengthen \eqref{eq:Hardy_interp} by proving the left-hand side
inequality for any $p\in[1,+\infty)$ and by providing a simple and
direct proof of the right-hand side one; \smallskip{}
\item prove conjecture \eqref{eq:Conjecture} for all radial-weighted Bergman
spaces $X=A^{p}(\beta)$ (see Theorem \ref{thm:Interp} below); \smallskip{}
\item apply Theorem \ref{thm:Interp} to spectral estimates
on norms of functions of matrices (see Subsection \ref{sub:motivation}
for details and Corollary \ref{cor:matrix} for the corresponding
statement); \smallskip{}
\item show that the embedding of $H^{\infty}$ into $A^{p}(\beta)$ is invertible
on the subset of rational functions of a given degree $n$ whose poles
are separated from the unit circle and obtain an asymptotically sharp
estimate for the embedding constant (see Subsection \ref{subsec:Problem-2.-Embedding}
for details and Theorem \ref{thm:Embedd} for the corresponding statement). 
\end{enumerate}
\medskip{}


\subsection{\label{sub:motivation} Motivations from matrix analysis }

Let $\cM_{n}$ be the set of complex $n\times n$ matrices and let
$\Norm{T}{}$ denote the operator norm of $T\in\cM_{n}$ associated
with the Hilbert norm on $\mathbb{C}^{n}$. We denote by $\sigma=\sigma(T)$
the spectrum of $T$, by $m_{T}$ its minimal polynomial, and by $\abs{m_{T}}$
the degree of $m_{T}$. In our discussion we will assume that $\Norm{T}{}\leq1$
and call such $T$ a \textit{contraction}. Let $\cC_{n}\subset\cM_{n}$
denote the set of all contractions. For a finite sequence $\sigma$
in $\mathbb{D}$, we denote by $P_{\sigma}$ the monic polynomial
with zero set $\sigma$ (counted with multiplicities). For a finite
sequence $\sigma$ in $\mathbb{D}$ and $f\in\mathcal{H}ol(\mathbb{D}),$
V. Pt\'ak and N. Young \cite{PtYo} introduced the quantity 
\[
\mathcal{M}(f,\,\sigma)=\mbox{sup}\left\{ \Norm{f(T)}{}:\,T\in\cC_{n},\,m_{T}=P_{\sigma}\right\} .
\]
Note that interesting cases occur for $f$ such that: \smallskip{}

\begin{enumerate}
\item $f\vert_{\sigma}={\displaystyle z^{k}\vert_{\sigma}}$ (estimates
on the norm of the powers of an $n\times n$ matrix, see for example
\cite{Pta}); \smallskip{}
\item $f\vert_{\sigma}=z^{-1}\vert_{\sigma}$ (estimates on condition numbers
and the norm of inverses of $n\times n$ matrices, see \cite{Nik0});
\smallskip{}
\item $f\vert_{\sigma}=(\zeta-z)^{-1}\vert_{\sigma}$ (estimates on the
norm of the resolvent of an $n\times n$ matrix, see for example \cite{Nik0,SzZa}). 
\end{enumerate}
\smallskip{}

Given a Blaschke sequence $\sigma$ in $\mathbb{D}$ and $f\in H^{\infty}$
it is possible to evaluate $\mathcal{M}(f,\,\sigma)$ as follows:
\begin{equation}
\mathcal{M}(f,\,\sigma)=\Norm{f}{H^{\infty}/B_{\sigma}H^{\infty}}=\Norm{f(M_{B_{\sigma}})}{},\label{eq:Sarason}
\end{equation}
where $M_{B_{\sigma}}$ is the compression of the multiplication operation
by $z$ to the model space $K_{B_{\sigma}}$, see Subsection \ref{sub:Model-spaces-and}
for the definitions. This formula is due to N. K. Nikolski \cite[Theorem 3.4]{Nik0}
while the last equality is a well-known corollary of Commutant Lifting
Theorem of B. Sz.-Nagy and C. Foia\c{s} \cite{NaFo,FoFr,Sar}.

Let $X\subset\mathcal{H}ol(\mathbb{D})$ be a Banach space containing
$H^{\infty}$. 
The above equality on $\mathcal{M}(f,\,\sigma)$ naturally extends
to any $f\in X$ as follows. There exists an analytic polynomial $p$
interpolating $f$ on the finite set $\sigma$. Therefore for any
$T\in\cC_{n}$ with $m_{T}=P_{\sigma}$ and $\sigma\subset\mathbb{D}$,
we have $f(T)=p(T)$ (since $f=p+m_{T}h$ for some $h\in\mathcal{H}ol(\mathbb{D})$).
Hence, 
\begin{eqnarray*}
\mathcal{M}(f,\,\sigma) & = & \mathcal{M}(p,\,\sigma)=\Norm{p}{H^{\infty}/B_{\sigma}H^{\infty}}\\
 & = & \Norm{p(M_{B_{\sigma}})}{}=\Norm{f(M_{B_{\sigma}})}{}.
\end{eqnarray*}
Here we used \eqref{eq:Sarason} applied to $p$. Moreover 
\begin{align*}
\Norm{p}{H^{\infty}/B_{\sigma}H^{\infty}} & =\inf\{\Norm{p+B_{\sigma}h}{\infty}:\:h\in H^{\infty}\}\\
 & =\inf\{\Norm{g}{\infty}:\:\:g\vert_{\sigma}=p\vert_{\sigma},\:g\in H^{\infty}\}\\
 & =\inf\{\Norm{g}{\infty}:\:\:g\vert_{\sigma}=f\vert_{\sigma},\:g\in H^{\infty}\}.
\end{align*}
We conclude that 
\[
\mathcal{M}(f,\,\sigma)=\inf\,\big\{\Norm{g}{\infty}:\:g\in H^{\infty},\:g\vert_{\sigma}=f\vert_{\sigma}\big\}.
\]
Therefore, given a division-closed Banach space $X\subset\mathcal{H}ol\left(\mathbb{D}\right)$
containing $H^{\infty}$ and a finite sequence $\sigma$ in $\mathbb{D}$,
it turns out that 
\[
\sup_{\Norm{f}{X}\leq1}\mathcal{M}(f,\,\sigma)=c\left(\sigma,\,X,\,H^{\infty}\right)
\]
and so, for all $n\geq1,$ $r\in(0,\,1)$, 
\[
\sup_{\Norm{f}{X}\leq1}\big\{\mathcal{M}(f,\,\sigma):\:\sigma\in\mathbb{D}^{n},\:\max_{\lambda\in\sigma}\left|\lambda\right|\leq r\big\}=\mathcal{C}_{n,\,r}(X,\,H^{\infty}).
\]
\smallskip{}


\subsection{\label{subsec:Problem-2.-Embedding} Embedding theorems for rational
functions}

In \cite{Dyn1,Dyn2} the following phenomenon was discovered: sharp
embedding theorems are invertible on the set of rational functions
of a given degree. Let $n\geq1$, let $\mathcal{P}_{n}$ be the space
of complex analytic polynomials of degree less or equal than $n$
and let 
\[
\mathcal{R}_{n}=\{P/Q:\;P,\,Q\in\mathcal{P}_{n},\;Q(\zeta)\ne0\text{\:\ for}\:\abs{\zeta}\leq1\}
\]
be the set of rational functions of degree at most $n$ with poles
outside of the closed unit disc $\overline{\mathbb{D}}=\{z\in\mathbb{C}:\,\vert z\vert\leq1\}$.
Recall that the Hardy--Littlewood embedding theorem \cite[Theorem 1.1]{Dur}
says that $H^{q}\subset A^{p}(\beta)$ for any $p>1$, $\beta>-1$
and $q=\frac{p}{2+\beta}$. Given two Banach spaces of analytic functions
in the disc which contain $\mathcal{R}_{n}$, denote by $\mathcal{E}_{n}(X,\,Y)$
the best possible constant such that 
\begin{equation}
\Norm{f}{X}\leq\mathcal{E}_{n}(X,\,Y)\Norm{f}{Y},\qquad f\in\mathcal{R}_{n}.\label{eq:embedd}
\end{equation}
Dyn'kin \cite[Theorem 4.1]{Dyn2} proved that the Hardy--Littlewood
embedding theorem is invertible on $\mathcal{R}_{n}$; namely, 
\[
\mathcal{E}_{n}(H^{q},\,A^{p}(\beta))\asymp n^{\frac{1+\beta}{p}},
\]
when $q=\frac{p}{2+\beta}$.

Note that for many choices of $X$ and $Y$ we have $\mathcal{E}_{n}(X,\,Y)=+\infty$
for every $n\in\mathbb{N}$, since the poles of the rational functions
are allowed to be arbitrarily close to the unit circle $\mathbb{T}=\{z\in\mathbb{C}:\,\vert z\vert=1\}$.
This is for example the case when $X=H^{\infty}$ and $Y=H^{p}$ or
$Y=A^{p}$, $1\leq p<+\infty$ (to see this one can consider the function
$f(z)=(1-rz)^{-1}$ as $r\rightarrow1^{-}$). This observation suggests
to consider a more general problem when one replaces the class $\mathcal{R}_{n}$
in \eqref{eq:embedd} by $\mathcal{R}_{n,\,r}$ (for any fixed $r\in[0,\,1)$)
defined by 
\[
\mathcal{R}_{n,\,r}=\Big\{ P/Q:\;P,\,Q\in\mathcal{P}_{n},\;Q(\zeta)\ne0\;\;{\rm for}\;\;\abs{\zeta}<\frac{1}{r}\Big\},
\]
i.e., by the set of all rational functions of degree at most $n\ge1$
without poles in $\frac{1}{r}\mathbb{D}$. The quantity $\mathcal{E}_{n}(X,\,Y)$
(when it is infinite) is replaced by the best possible constant $\mathcal{E}_{n,\,r}(X,\,Y)$
such that 
\[
\Norm{f}{X}\leq\mathcal{E}_{n,\,r}(X,\,Y)\Norm{f}{Y},\qquad f\in\mathcal{R}_{n,\,r},
\]
and we can study the asymptotic dependence on the parameters $r$
and $n$ as $r\to1^{-}$ and $n\to\infty$. Recently, the authors
\cite[Theorem 2.4]{BaZa1} proved S.\,M.~Nikolskii-type inequalities
for rational functions (whose poles do not belong to $\mathbb{T}$)
which can be formulated here as 
\begin{equation}
\mathcal{E}_{n,\,r}(H^{q},\,H^{p})\asymp\bigg(\frac{n}{1-r}\bigg)^{\frac{1}{p}-\frac{1}{q}},\qquad0\leq p<q\leq\infty.\label{eq:BaZaNik}
\end{equation}
\smallskip{}

\subsection{Outline of the paper}

The paper is organized as follows. Section \ref{sec:Known-results-and}
states our main results. Section \ref{sec:Main-ingredients} is devoted
to the main ingredients and tools employed in the proofs. In particular,
we recall the so-called \textit{theory of model spaces} which plays
a central role here (see Subsection \ref{sub:Model-spaces-and}) and
discuss the strategy of the proofs of main results (Subsection \ref{str}).
Section \ref{est:norms} contains sharp asymptotic estimates for the
norms of derivatives of reproducing kernels in various function spaces.
In Sections \ref{sec:Proofs} and \ref{sec:Proof-of-the} we prove,
respectively, the upper and the lower bounds in Theorems \ref{thm:Interp}
and \ref{thm:Embedd}. \bigskip{}


\section{\label{sec:Known-results-and} Main results}
\begin{thm}
\label{thm:Interp} Let $n\geq1$, $r\in[0,\,1)$, $p\in[1,\,+\infty)$
and $\beta>-1$. Then we have 
\begin{equation}
C_{n,\,r}\left(H^{p},\,H^{\infty}\right)\asymp\left(\frac{n}{1-r}\right)^{\frac{1}{p}}\label{eq:strength_hardy}
\end{equation}
with constants depending only on $p$, and 
\begin{equation}
C_{n,\,r}\left(A^{p}(\beta),\,H^{\infty}\right)\asymp\bigg(\frac{n}{1-r}\bigg)^{\frac{2+\beta}{p}}\label{eq:strength_bergm}
\end{equation}
with constants depending only on $p$ and $\beta$. 
\end{thm}

In view of the discussion in Subsection \ref{sub:motivation}, the
following corollary is immediate: 
\begin{cor}
\label{cor:matrix} Let $n\geq1$, $r\in[0,\,1)$, $p\in[1,\,+\infty)$
and $\beta>-1$. Then we have 
\[
\sup_{\Norm{f}{H^{p}}\leq1}\big\{\mathcal{M}(f,\,\sigma):\:\sigma\in\mathbb{D}^{n},\:\max_{\lambda\in\sigma}\left|\lambda\right|\leq r\big\}\asymp\left(\frac{n}{1-r}\right)^{\frac{1}{p}}
\]
with constants depending only on $p$, and 
\[
\sup_{\Norm{f}{A^{p}(\beta)}\leq1}\big\{\mathcal{M}(f,\,\sigma):\:\sigma\in\mathbb{D}^{n},\:\max_{\lambda\in\sigma}\left|\lambda\right|\leq r\big\}\asymp\bigg(\frac{n}{1-r}\bigg)^{\frac{2+\beta}{p}}
\]
with constants depending only on $p$ and $\beta$. 
\end{cor}

The techniques employed to prove Theorem \ref{thm:Interp} make use
of the theory of model spaces and their reproducing kernels. They
naturally lead to asymptotically sharp estimates of the embedding
constants $\mathcal{E}_{n,\,r}(H^{\infty},\,A^{p}(\beta))$, $1\leq p<\infty$,
$\beta>-1$. 
\begin{thm}
\label{thm:Embedd} Let $n\geq1$, $r\in[0,\,1)$, $p\in[1,\,+\infty)$
and $\beta>-1$. Then we have 
\begin{equation}
\mathcal{E}_{n,\,r}(H^{\infty},\,A^{p}(\beta))\asymp\bigg(\frac{n}{1-r}\bigg)^{\frac{2+\beta}{p}}\label{eq:embedd_bergman}
\end{equation}
with constants depending only on $p$ and $\beta$. 
\end{thm}

\bigskip{}


\section{\label{sec:Main-ingredients} Main ingredients}

In this section we give the main ingredients and tools we use in the
proofs of Theorem~\ref{thm:Interp} and Theorem~\ref{thm:Embedd}.
We begin with the definition of model spaces. \smallskip{}

\subsection{\label{sub:Model-spaces-and} Model spaces}

Let $\Theta$ be an \textit{inner function}, i.e., $\Theta\in H^{\infty}$
and $\vert\Theta(\xi)|=1$ for a.e. $\xi\in\mathbb{T}$. We define
the model subspace $K_{\Theta}$ of the Hardy space $H^{2}$ by 
\[
K_{\Theta}=H^{2}\cap\left(\Theta H^{2}\right)^{\perp}=H^{2}\ominus\Theta H^{2}.
\]
By the famous theorem of Beurling, these and only these subspaces
of $H^{2}$ are invariant with respect to the backward shift operator
$S^{*}$ defined by 
\[
S^{*}f=\frac{f-f(0)}{z}.
\]
We refer to \cite{Nik1} for the general theory of the spaces $K_{\Theta}$
and their numerous applications. Given $\sigma\in\mathbb{D}^{n}$,
put $B=B_{\sigma}$ and consider the model subspace $K_{B}$. Let
us first establish the relation between $\mathcal{R}_{n},$ $\mathcal{R}_{n,\,r},$
and model spaces $K_{B}$. It is well known that if 
\[
\sigma=(\lambda_{1},...,\lambda_{1},\lambda_{2},...,\lambda_{2},...,\lambda_{t},...,\lambda_{t})\in\mathbb{D}^{n},
\]
where every $\lambda_{s}$ is repeated according to its multiplicity
$n_{s}$, $\sum_{s=1}^{t}n_{s}=n$, then 
\[
K_{B}=\overline{{\rm span}}\{k_{\lambda_{s},\,j}:\,1\leq s\leq t,\,0\leq j\leq n_{s}-1\},
\]
where for $\lambda\neq0$, $k_{\lambda,\,j}=\left(\frac{d}{d\overline{\lambda}}\right)^{j}k_{\lambda}$
and $k_{\lambda}(z)=\frac{1}{1-\overline{\lambda}z}$ is the standard
Cauchy kernel at the point $\lambda,$ whereas $k_{0,\,j}=z^{j}$.
Thus the subspace $K_{B}$ consists of rational functions of the form
$P/Q$, where $P\in\mathcal{P}_{n-1}$ and $Q\in\mathcal{P}_{n}$,
with the poles $1/\overline{\lambda}_{1},\dots,1/\overline{\lambda}_{n}$
of corresponding multiplicities (including possible poles at $\infty$).
Hence, if $f\in\mathcal{R}_{n}$ and $1/\overline{\lambda}_{1},\dots,1/\overline{\lambda}_{n}$
are the poles of $f$, then $f\in K_{zB}$ with $\sigma=(\lambda_{1},\dots,\lambda_{n})$.

For any inner function $\Theta$ the reproducing kernel of the model
space $K_{\Theta}$ corresponding to a point $\zeta\in\mathbb{D}$
is of the form 
\[
k_{\zeta}^{\Theta}(z)=
\frac{1-\overline{\Theta(\zeta)}\Theta(z)}{1-\overline{\zeta}z}=
(1-\overline{\Theta(\zeta)}\Theta(z))k_{\zeta}(z).
\]
We recall the definition of the Malmquist--Walsh family
$(e_{j})_{j=1}^{n}$ for a sequence $\sigma=(\lambda_{1},\dots,\lambda_{n})\in\mathbb{D}^{n}$
(see \cite[p. 117]{Nik2}): 
\[
e_{1}=\left(1-\vert\lambda_{1}\vert^{2}\right)k_{\lambda_{1}},\qquad e_{j}=\left(1-\abs{\lambda_{j}}^{2}\right)^{\frac{1}{2}}\left(\prod_{i=1}^{j-1}b_{\lambda_{i}}\right)k_{\lambda_{j}},\quad\,j=2\dots n.
\]
Note that $(e_{j})_{j=1}^{n}$ is an orthonormal basis of $K_{B}$
for $B=B_{\sigma}$. The \emph{model operator} $M_{B}$ evoked above
in Section \ref{sub:motivation} is the compression of the shift operator
$S:\,f\mapsto zf$ on $K_{B}$, i.e., $M_{B}f=P_{B}Sf$, $f\in K_{B}$,
where $P_{B}$ is the orthogonal projection on $K_{B}$.


\subsection{Upper bounds in Theorems \ref{thm:Interp} and \ref{thm:Embedd}.}

\label{str} In this subsection we outline the strategy of the proof
of Theorems \ref{thm:Interp} and \ref{thm:Embedd}.

From now on we denote by $\left\langle \cdot,\,\cdot\right\rangle $
the Cauchy sesquilinear form: 
\[
\left\langle h,\,g\right\rangle =\sum_{k\geq0}\hat{h}(k)\overline{\hat{g}(k)},
\]
which makes sense for any $h=\sum_{k\geq0}\hat{h}(k)z^{k}\in\mathcal{H}ol(\mathbb{D})$
and $g=\sum_{k\geq0}\hat{g}(k)z^{k}$ analytic in the disc $(1+\delta)\mathbb{D}$
for some $\delta>0$. If $h,\,g\in H^{2}$, then $\left\langle \cdot,\,\cdot\right\rangle $
coincides with the usual scalar product in $L^{2}(\mathbb{T})$, 
\[
\left\langle h,\,g\right\rangle =\int_{\mathbb{T}}h(u)\overline{g(u)}{\rm d}m(u),
\]
where $m$ is the normalized Lebesgue measure on $\mathbb{T}$. Also,
denote by $(h,\,g)$ the scalar product on $A^{2}$ defined by 
\[
(h,\,g)=\int_{\mathbb{D}}h(u)\overline{g(u)}{\rm d}\mathcal{A}(u),\qquad h,\,g\in A^{2}.
\]

As in \cite{Zar1,Zar2}, we will use the following interpolation operator:
\begin{equation}
f\mapsto P_{B}f=\sum_{k=1}^{n}\left\langle f,\,e_{k}\right\rangle e_{k},\label{eq:interpolator}
\end{equation}
where $\left(e_{k}\right)_{k=1}^{n}$ is the Malmquist--Walsh
basis of $K_{B}$. If $f\in H^{2}$, then this is the usual orthogonal
projection of $f$ onto $K_{B}$. However the formula $P_{B}f=\sum_{k=1}^{n}\left\langle f,\,e_{k}\right\rangle e_{k}$
correctly defines this operator for any $f\in\mathcal{H}ol(\mathbb{D})$.
\smallskip{}

\subsubsection{The upper bounds in Theorem \textup{\ref{thm:Interp} }}

For $X=H^{p}$, $1\leq p<+\infty$, the proof is simple. We have 
\[
\abs{P_{B}f(\zeta)}=\abs{\left\langle f,\,k_{\zeta}^{B}\right\rangle }\leq\Norm{f}{H^{p}}\Norm{k_{\zeta}^{B}}{H^{q}},
\]
where $q$ is the conjugate of $p$. It remains to use the estimate
for $\Norm{k_{\zeta}^{B}}{H^{q}}$, $\zeta\in\mathbb{T}$, given in
Proposition \ref{lem:rep_kern_norms_est}. \smallskip{}

To relate $P_{B}f(\zeta)=\left\langle f,\,k_{\zeta}^{B}\right\rangle $
to the norm of $f$ in a Bergman space $A^{p}(\beta)$, we use the
simplest form of the Green formula, 
\begin{equation}
\left\langle \phi,\,\psi\right\rangle =(\phi',\,S^{*}\psi)+\phi(0)\overline{\psi(0)},\label{gr}
\end{equation}
which is true, in particular, when $\phi$ is analytic in some disc
$(1+\delta)\mathbb{D}$, $\delta>0$, and $\psi\in H^{\infty}$. We
then apply it to $\phi=k_{\zeta}^{B}$ and $\psi=f\in X\cap H^{\infty}$.

If $-1<\beta\le0$, then we apply the H\"older inequality and it remains
to estimate the norm of the derivative $(k_{\zeta}^{B})'$ in the
dual Bergman space. This norm is estimated in Proposition \ref{lem:rep_kern_norms_est}.
\smallskip{}

To treat the case $\beta>0$, we need a modified Green formula. Recall
that the fractional differentiation operator $D_{\alpha}$, $-1<\alpha<\infty$,
is defined by $D_{\alpha}(z^{m})=\frac{\Gamma(m+2+\alpha)}{(m+1)!\Gamma(2+\alpha)}z^{m}$,
$m=0,\,1,\,2,\,\dots$, and extended linearly to the whole space $\mathcal{H}ol(\mathbb{D})$
(see \cite[Lemma 1.17]{HKZ}). Then, for a function $f$ analytic
in a neighborhood of $\overline{\mathbb{D}}$ and $-1<\alpha<\infty$,
we have 
\begin{equation}
\int_{\mathbb{D}}f(u)\overline{g(u)}{\rm d}\mathcal{A}(u)=(\alpha+1)\int_{\mathbb{D}}D_{\alpha}f(u)\overline{g(u)}\left(1-\left|u\right|^{2}\right)^{\alpha}{\rm d}\mathcal{A}(u)\label{gr2}
\end{equation}
for any $g\in H^{\infty}$ (see \cite[Lemma 1.20]{HKZ}). Note that
even for $l\in\mathbb{N}$, $D_{l}f$ differs from the usual derivative
$f^{(l)}$. However, 
\begin{equation}
\|D_{l}f\|_{A^{p}(\beta)}\asymp\|f^{(l)}\|_{A^{p}(\beta)}+\sum_{j=0}^{l-1}|f^{(j)}(0)|\label{eq:}
\end{equation}
for any Bergman space $A^{p}(\beta)$.

Formula \eqref{gr2} reduces the problem to estimates of the Bergman
norms of $(k_{\zeta}^{B})^{(l)}$, $l\in\mathbb{N}$, which are again
given in Proposition \ref{lem:rep_kern_norms_est}. \smallskip{}

\subsubsection{The upper bound in Theorem\textup{ \ref{thm:Embedd}} }

To prove the upper bound in Theorem \ref{thm:Embedd}
it is sufficient to note that given $f\in\mathcal{R}_{n,\,r}$ with
poles $1/\overline{\lambda}_{1},\dots,1/\overline{\lambda}_{n}$ (repeated
according to multiplicities and satisfying $\left|\lambda_{i}\right|\leq r$
for all $i=1\dots n$), we have $f\in K_{zB}$ where $B=B_{\sigma}$,
$\sigma=(\lambda_{1},\dots,\lambda_{n})$. Therefore 
\[
f(\zeta)=\left\langle f,\,k_{\zeta}^{zB}\right\rangle .
\]
This means that $f$ pointwise coincides with the interpolation operator
\eqref{eq:interpolator} and we can apply the same reasoning as above
with $zB$ instead of $B$. \smallskip{}


\subsection{Lower bounds }

The lower bound problem in Theorem \ref{thm:Interp} is treated by
using the ``worst'' interpolation $n$-tuple $\sigma=\sigma_{\lambda,\,n}=(\lambda,\,...,\,\lambda)\in\mathbb{D}^{n}$,
a one-point set of multiplicity $n$ (a Carath\'eodory--Schur type
interpolation problem). The ``worst'' interpolation data comes from
the Dirichlet kernels $\sum_{k=0}^{n-1}z^{k}$ transplanted from the
origin to $\lambda$. The lower bound in \eqref{eq:embedd_bergman}
is achieved by rational functions of the same kind (i.e., whose poles
are concentrated at the same point $1/\bar{\lambda}$). \bigskip{}


\section{\label{est:norms} Estimates of norms of reproducing kernels}

We will use the following simple Bernstein-type inequality for rational
functions (see, e.g., \cite[Theorem 2.3]{BaZa1}). Let $\sigma=(\lambda_{1},\dots,\lambda_{n})\in\mathbb{D}^{n}$,
let $B=B_{\sigma}$ be the corresponding finite Blaschke product and
$r=\max_{\lambda\in\sigma}\vert\lambda\vert$. Given $l\in\mathbb{N}$
we have 
\begin{equation}
\Norm{f^{(l)}}{\infty}\lesssim\left(\frac{n}{1-r}\right)^{l}\Norm{f}{\infty}\label{prop:Bernst_type_ineq}
\end{equation}
for any $f\in K_{B}$.

We need to introduce an additional scale of Banach spaces of holomorphic
functions in $\mathbb{D}$. The weighted Bloch space $\mathcal{B}_{\alpha}$,
$0\leq\alpha\leq1$, consists of functions $f\in\mathcal{H}ol(\mathbb{D})$
satisfying 
\[
\Norm{f}{\mathcal{B}_{\alpha}}=\sup_{z\in\mathbb{D}}\left|f'(z)\right|\left(1-\left|z\right|\right)^{\alpha}<\infty
\]
(which is in fact a seminorm). 
\begin{prop}
\label{lem:rep_kern_norms_est}Let $\sigma=(\lambda_{1},\dots,\lambda_{n})\in\mathbb{D}^{n}$
and $B=B_{\sigma}$ be the corresponding finite Blaschke product,
$r=\max_{\lambda\in\sigma}\vert\lambda\vert$ and $z\in\overline{\mathbb{D}}$.
The following inequalities hold: 
\begin{enumerate}
\item Given $q\in(1,\,+\infty]$ we have 
\begin{equation}
\Norm{k_{z}^{B}}{H^{q}}\lesssim\left(\frac{n}{1-r}\right)^{1-\frac{1}{q}}.\label{eq:Hpnorms}
\end{equation}
\item Given $l\in\mathbb{N}$ and \textup{$0\leq\alpha\leq1$} we have 
\begin{equation}
\Norm{(k_{z}^{B})^{(l)}}{\mathcal{B}_{\alpha}}\lesssim\left(\frac{n}{1-r}\right)^{l+2-\alpha}.\label{eq:Bloch_norms}
\end{equation}
\item Given $q\in(1,\,+\infty)$ and $\gamma\in(-1,\,q-1]$ we have 
\begin{equation}
\Norm{(k_{z}^{B})'}{A^{q}(\gamma)}\lesssim\left(\frac{n}{1-r}\right)^{2-\frac{\gamma+2}{q}}.\label{eq:1st_estimate_AntonLemma}
\end{equation}
\item Given $q\in(1,\,+\infty)$ and $\gamma\in(-1,\,q]$ we have 
\begin{equation}
\Norm{(k_{z}^{B})''}{A^{q}(\gamma)}\lesssim\left(\frac{n}{1-r}\right)^{3-\frac{\gamma+2}{q}}.\label{eq:2nd_estimate_AntonLemma}
\end{equation}
\end{enumerate}
All involved constants may depend on $q$, $\alpha$ and $\gamma$
but do not depend on $n$, $r$ and $z$. 
\end{prop}

\begin{proof}
Note that all above norms (in appropriate powers) are subharmonic
functions in $z$. Thus, it suffices to prove the inequalities only
in the case $z=\zeta\in\mathbb{T}$. So in what follows we assume
that $\zeta\in\mathbb{T}$. 
\medskip{}
\\
\textbf{Proof of \eqref{eq:Hpnorms}.} For $q=2$, 
\begin{eqnarray*}
\Norm{k_{\zeta}^{B}}{H^{2}}^{2} & = & \frac{1-\abs{B(\zeta)}^{2}}{1-\abs{\zeta}^{2}}=\sum_{j=1}^{n}\frac{1-\abs{\lambda_{j}}^{2}}{\abs{1-\overline{\lambda_{j}}\zeta}^{2}}\abs{B_{j-1}(\zeta)}^{2}\\
 & \leq & \sum_{j=1}^{n}\frac{1+\abs{\lambda_{j}}}{1-\abs{\lambda_{j}}}\leq\frac{1+r}{1-r}n.
\end{eqnarray*}
Here $B_{j}=\prod_{i=1}^{j}b_{\lambda_{i}}$, $B_{0}\equiv1$. The
estimate for $q=\infty$ follows from 
\[
\abs{k_{\zeta}^{B}(u)}\leq\Norm{k_{\zeta}^{B}}{H^{2}}\Norm{k_{u}^{B}}{H^{2}}\lesssim\frac{n}{1-r}
\]
for any $u,\,\zeta\in\overline{\mathbb{D}}$. If $q\in[2,\,\infty)$,
then we write 
\[
\Norm{k_{\zeta}^{B}}{H^{q}}^{q}\leq\Norm{k_{\zeta}^{B}}{H^{2}}^{2}\Norm{k_{\zeta}^{B}}{\infty}^{q-2}\lesssim\left(\frac{n}{1-r}\right)^{q-1}.
\]
Finally for $q\in(1,\,2)$ we apply the following result by W. Cohn
\cite[Lemma 4.2]{Coh}. Denoting by $p$ the conjugate exponent of
$q$ (i.e., $\frac{1}{q}+\frac{1}{p}=1$) 
\begin{eqnarray*}
\Norm{k_{\zeta}^{B}}{H^{q}} & = & \sup_{g\in H^{p},\:\Norm{g}{H^{p}}\leq1}\Abs{\int_{\mathbb{T}}g(z)\overline{k_{\zeta}^{B}(z)}{\rm d}m(z)}\\
 & = & \sup_{g\in H^{p},\:\Norm{g}{H^{p}}\leq1}\Abs{\int_{\mathbb{T}}P_{B}g(z)\overline{k_{\zeta}^{B}(z)}{\rm d}m(z)}\\
 & \lesssim & \sup_{h\in K_{B},\:\Norm{h}{H^{p}}\leq C_{p}}\abs{h(\zeta)},
\end{eqnarray*}
where the last inequality is due to the fact that $h=P_{B}g\in K_{B}$
and there exists $C_{p}>0$ 
such that 
$\Norm{P_{B}g}{H^{p}}\leq C_{p}\Norm{g}{H^{p}}$, $p\in(1,\,\infty)$.

Applying the inequality 
$\|h\|_\infty \le \big(\frac{n}{1-r}\big)^{1/p} \|h\|_{H^p}$, 
$h\in K_B$, which is a special case of \eqref{eq:BaZaNik}, we obtain
\eqref{eq:Hpnorms}. \medskip{}
 \\
 \textbf{Proof of \eqref{eq:Bloch_norms}.} Clearly, for $\alpha=1$,
\[
\sup_{u\in\mathbb{D}}(1-\vert u\vert)\big|(k_{\zeta}^{B})^{(l+1)}(u)\big|=\Norm{(k_{\zeta}^{B})^{(l)}}{\mathcal{B}_{1}}\leq\Norm{(k_{\zeta}^{B})^{(l)}}{\infty}\lesssim\left(\frac{n}{1-r}\right)^{l+1}
\]
by \eqref{prop:Bernst_type_ineq}. Therefore, for any $0\leq\alpha\leq1$,
\[
\Norm{(k_{\zeta}^{B})^{(l)}}{\mathcal{B}_{\alpha}}\leq\Norm{(k_{\zeta}^{B})^{(l)}}{\mathcal{B}_{1}}^{\alpha}\Norm{(k_{\zeta}^{B})^{(l+1)}}{\infty}^{1-\alpha}\lesssim\left(\frac{n}{1-r}\right)^{l+2-\alpha},
\]
which completes the proof. \medskip{}
 \\
 \textbf{Proof of \eqref{eq:1st_estimate_AntonLemma}. } For the derivative
of $k_{\zeta}^{B}$ we have 
\[
(k_{\zeta}^{B})'(z)=\zeta\overline{B(\zeta)}\frac{B(\zeta)-B(z)-(\zeta-z)B'(z)}{(\zeta-z)^{2}}.
\]
Then we can write $\Norm{(k_{\zeta}^{B})'}{A^{q}(\gamma)}^{q}=I_{1}+I_{2}$,
where 
\[
I_{1}=\int_{\abs{z-\zeta}\leq\frac{1-r}{n}}\Abs{\frac{B(\zeta)-B(z)-(\zeta-z)B'(z)}{(\zeta-z)^{2}}}^{q}(1-\vert z\vert)^{\gamma}{\rm d}\mathcal{A}(z)
\]
and 
\[
I_{2}=\int_{\abs{z-\zeta}>\frac{1-r}{n}}\Abs{\frac{B(\zeta)-B(z)-(\zeta-z)B'(z)}{(\zeta-z)^{2}}}^{q}(1-\vert z\vert)^{\gamma}{\rm d}\mathcal{A}(z).
\]
Since $|B''(u)| \lesssim\big(\frac{n}{1-r}\big)^{2}$ for any
$u\in\overline{\mathbb{D}}$, it follows that 
\begin{align*}
I_{1} & \lesssim\max_{u\in\overline{\mathbb{D}}}\abs{B''(u)}^{q}\int_{\abs{z-\zeta}\leq\frac{1-r}{n}}(1-\vert z\vert)^{\gamma}{\rm d}\mathcal{A}(z)\\
 & \lesssim\left(\frac{n}{1-r}\right)^{2q}\left(\frac{1-r}{n}\right)^{\gamma+2}\lesssim\left(\frac{n}{1-r}\right)^{2q-2-\gamma}.
\end{align*}

Now we estimate $I_{2}$. To this aim we first observe that if we
put $w=\big(1-\frac{1-r}{2n}\big) \zeta$, then $|z-\zeta|/2\le|1-\bar{w}z|\le3|z-\zeta|/2$
when $|z-\zeta|\ge(1-r)/n$. Hence, 
\[
\begin{aligned}\int_{\abs{z-\zeta}>\frac{1-r}{n}}
\Abs{\frac{B(\zeta)-B(z)}{(z-\zeta)^{2}}}^{q}(1-\vert z\vert)^{\gamma} & 
{\rm d}\mathcal{A}(z) \lesssim \int_{\mathbb{D}}\frac{(1-\vert z\vert)^{\gamma}}{\abs{1-\bar{w}z}^{2q}}{\rm d}\mathcal{A}(z)\\
 & \asymp\frac{1}{(1-|w|)^{2(q-1)-\gamma}}\asymp\left(\frac{n}{1-r}\right)^{2q-\gamma-2}.
\end{aligned}
\]
Here we use the standard fact (see, e.g., \cite[Theorem 1.7]{HKZ})
that for $\alpha>-1$ and $\beta>\alpha+2$ one has 
\begin{equation}
\int_{\mathbb{D}}\frac{(1-\vert z\vert)^{\alpha}}{\abs{1-\bar{w}z}^{\beta}}{\rm d}\mathcal{A}(z)\asymp\frac{1}{(1-|w|)^{\beta-\alpha-2}}\label{bas}
\end{equation}
with constants depending on $\alpha$ and $\beta$, but not on $w\in\mathbb{D}$.
Note that $\gamma\in(-1,\,q-1]$ and so in our case the assumptions
on exponents are satisfied.

It remains to estimate 
\[
\int_{\abs{z-\zeta}>\frac{1-r}{n}}\Abs{\frac{B'(z)}{\zeta-z}}^{q}(1-\vert z\vert)^{\gamma}{\rm d}\mathcal{A}(z)\asymp\int_{\mathbb{D}}\frac{\abs{B'(z)}^{q}}{\abs{1-\bar{w}z}^{q}}(1-\vert z\vert)^{\gamma}{\rm d}\mathcal{A}(z)=:J.
\]
Take $\epsilon>0$ sufficiently small so that $\epsilon<\min(q-1,\gamma+1)$,
and put $s=\gamma+1-\epsilon$. Then $s\in(0,q)$. Writing 
\[
\abs{B'(z)}^{q}(1-\vert z\vert)^{\gamma}=\abs{B'(z)}^{q-s}(1-\vert z\vert)^{\gamma-s}\left(\abs{B'(z)}(1-\abs{z})\right)^{s}
\]
and observing that $\sup_{u\in\mathbb{D}}\abs{B'(u)}(1-\abs{u})\leq1$,
we get 
\[
J\leq\int_{\mathbb{D}}\frac{\abs{B'(z)}^{q-s}(1-\vert z\vert)^{\gamma-s}}{\abs{1-\bar{w}z}^{q}}{\rm d}\mathcal{A}(z)\lesssim\left(\frac{n}{1-r}\right)^{q-s}\int_{\mathbb{D}}\frac{(1-\vert z\vert)^{\gamma-s}}{\abs{1-\bar{w}z}^{q}}{\rm d}\mathcal{A}(z)
\]
(here we use the inequality $\abs{B'(u)}\lesssim\frac{n}{1-r}$, $u\in\overline{\mathbb{D}}$).
Since $\gamma-s=-1+\epsilon$ and $q-\gamma+s-2=q-1-\epsilon>0$,
we get by \eqref{bas} 
\[
\int_{\mathbb{D}}\frac{(1-\vert z\vert)^{\gamma-s}}{\abs{1-\bar{w}z}^{q}}{\rm d}\mathcal{A}(z)\asymp\frac{1}{(1-|w|)^{q-\gamma+s-2}}\asymp\left(\frac{n}{1-r}\right)^{q-\gamma+s-2},
\]
which completes the proof of \eqref{eq:1st_estimate_AntonLemma}. \medskip{}
 \\
 \textbf{Proof of \eqref{eq:2nd_estimate_AntonLemma}.} Note that
\[
\abs{(k_{\zeta}^{B})''(z)}=2\Abs{\frac{B(\zeta)-B(z)-(\zeta-z)B'(z)-\frac{(\zeta-z)^{2}}{2}B''(z)}{(\zeta-z)^{3}}},
\]
whence $\abs{(k_{\zeta}^{B})''(z)}\le\sup_{u\in\mathbb{D}}|B'''(u)|/3$,
$z\in\mathbb{D}$. Since $\abs{B'''(u)}\lesssim\big(\frac{n}{1-r}\big)^{3}$
for any $u\in\overline{\mathbb{D}}$, it follows that 
\[
\Abs{(k_{\zeta}^{B})''(u)}\lesssim\sup_{u\in\mathbb{D}}\abs{B'''(u)}\lesssim\left(\frac{n}{1-r}\right)^{3}.
\]
Therefore 
\begin{align*}
\int_{\abs{z-\zeta}\leq\frac{1-r}{n}}\Abs{(k_{\zeta}^{B})''(z)}^{q}(1-\vert z\vert)^{\gamma}{\rm d}\mathcal{A}(z) & \lesssim\left(\frac{n}{1-r}\right)^{3q}\int_{\abs{z-\zeta}\leq\frac{1-r}{n}}(1-\vert z\vert)^{\gamma}{\rm d}\mathcal{A}(z)\\
 & \lesssim\left(\frac{n}{1-r}\right)^{3q}\left(\frac{1-r}{n}\right)^{\gamma+2}.
\end{align*}

It remains to estimate 
\[
J_{1}:=\int_{\abs{z-\zeta}>\frac{1-r}{n}}\frac{(1-\vert z\vert)^{\gamma}}{\abs{\zeta-z}^{3q}}{\rm d}\mathcal{A}(z),\qquad J_{2}:=\int_{\abs{z-\zeta}>\frac{1-r}{n}}\frac{\abs{B'(z)}^{q}}{\abs{\zeta-z}^{2q}}(1-\vert z\vert)^{\gamma}{\rm d}\mathcal{A}(z)
\]
and 
\[
J_{3}:=\int_{\abs{z-\zeta}>\frac{1-r}{n}}\frac{\abs{B''(z)}^{q}}{\abs{\zeta-z}^{q}}(1-\vert z\vert)^{\gamma}{\rm d}\mathcal{A}(z).
\]
The estimate $J_{1}\asymp\big(\frac{n}{1-r}\big)^{3q-2-\gamma}$
follows immediately from \eqref{bas}. Now let $s=\frac{\gamma+1}{2}$.
Then $s\in[0,\,q]$ and we have 
\[
J_{2}\lesssim\sup_{u\in\mathbb{D}}\abs{B'(u)}^{q-s}\cdot\sup_{u\in\mathbb{D}}\left((1-\abs{u})^{s}\abs{B'(u)}^{s}\right)\cdot\int_{\abs{z-\zeta}>\frac{1-r}{n}}\frac{(1-\vert z\vert)^{\gamma-s}}{\abs{\zeta-z}^{2q}}{\rm d}\mathcal{A}(z).
\]
As in the proof of \eqref{eq:1st_estimate_AntonLemma}, let $w=\zeta\left(1-\frac{1-r}{2n}\right)$.
Then, by \eqref{bas}, 
\[
J_{2}\lesssim\left(\frac{n}{1-r}\right)^{q-s}\int_{\mathbb{D}}\frac{(1-\vert z\vert)^{\gamma-s}}{\abs{1-\bar{w}z}^{2q}}{\rm d}\mathcal{A}(z)\asymp\left(\frac{n}{1-r}\right)^{q-s}\frac{1}{(1-|w|)^{2q-\gamma+s-2}}.
\]
We used the fact that $2q-\gamma+s-2=2q-\frac{3}{2}-\frac{\gamma}{2}>0$
since $q\ge\gamma$ and $q>1$.

To estimate $J_{3}$ note that $\abs{B''(u)}\leq2\left(1-\abs{u}\right)^{-2}$,
$u\in\mathbb{D}$. Let $\epsilon\in(0,\min(q-1,\gamma+1))$ and put
$s=\frac{\gamma+1-\epsilon}{2}.$ Then $s\in(0,q)$, $\gamma-2s=-1+\epsilon$
and $q-\gamma+2s-2=q-1-\epsilon>0$. Now 
\begin{align*}
J_{3} & \leq\sup_{u\in\mathbb{D}}\abs{B''(u)}^{q-s}\cdot\sup_{u\in\mathbb{D}}\left((1-\abs{u})^{2s}\abs{B''(u)}^{s}\right)\cdot\int_{\abs{z-\zeta}>\frac{1-r}{n}}\frac{(1-\vert z\vert)^{\gamma-2s}}{\abs{1-\bar{w}z}^{q}}{\rm d}\mathcal{A}(z)\\
 & \asymp\left(\frac{n}{1-r}\right)^{2q-2s}\left(\frac{1}{1-\abs{w}}\right)^{q-\gamma+2s-2}\asymp\left(\frac{n}{1-r}\right)^{3q-\gamma-2}.
\end{align*}
This completes the proof of \eqref{eq:2nd_estimate_AntonLemma}. 
\end{proof}
\begin{cor}
\label{cor:Anton_Cor}Let $l\in\mathbb{N},$ $l\geq2$, $q\in(1,\infty)$
and $\gamma\in(-1,\,q]$. We have 
\[
\Norm{(k_{\zeta}^{B})^{(l)}}{A^{q}(\gamma)}\lesssim\left(\frac{n}{1-r}\right)^{l+1-\frac{\gamma+2}{q}}.
\]
\end{cor}

\begin{proof}
An application of \cite[Theorem 1.3]{BaZa2} yields 
\[
\Norm{(k_{\zeta}^{B})^{(l)}}{A^{q}(\gamma)}\lesssim\left(\frac{n}{1-r}\right)^{l-2}\Norm{(k_{\zeta}^{B})^{''}}{A^{q}(\gamma)}.
\]
Now the result follows from inequality \eqref{eq:2nd_estimate_AntonLemma}. 
\end{proof}
\bigskip{}


\section{\label{sec:Proofs} Proofs of the upper bounds in Theorems \ref{thm:Interp}
and \ref{thm:Embedd}}

\subsection{The upper bounds in (\ref{eq:strength_hardy}): a direct proof}

We start by giving an easier proof than the one in \cite[Theorem 2.3]{Zar2}
of the upper bound in \eqref{eq:strength_hardy} for the case $1\leq p\leq+\infty$.
The main drawback of the proof in \cite{Zar2} is that it makes use
of a strong interpolation result between Hardy spaces by P. Jones
\cite{Jon}. The proof below is a two-line corollary of $H^{p}$-norms
estimates of reproducing kernel of model spaces. 
\begin{proof}
For any $f\in H^{p}$, 
\[
\vert P_{B}f(\zeta)\vert=\abs{\left\langle f,\,k_{\zeta}^{B}\right\rangle }\leq\Norm{f}{H^{p}}\Norm{k_{\zeta}^{B}}{H^{q}},\qquad\zeta\in\overline{\mathbb{D}},
\]
where $\frac{1}{p}+\frac{1}{q}=1.$ Taking the supremum over all $\zeta\in\overline{\mathbb{D}}$,
we obtain from \eqref{eq:Hpnorms} that 
\[
c\left(\sigma,\,H^{p},\,H^{\infty}\right)\lesssim\left(\frac{n}{1-r}\right)^{\frac{1}{p}},
\]
for any $\sigma=(\lambda_{1},\dots,\lambda_{n})\in\mathbb{D}^{n}.$ 
\end{proof}

\subsection{The upper bound in \eqref{eq:strength_bergm}}

In the following proof, given $\sigma\in\mathbb{D}^{n}$ we will assume
first that $f\in H^{\infty}$ and bound $\Norm{f}{H^{\infty}/B_{\sigma}H^{\infty}}$
in terms of $\Norm{f}{A^{p}\left(\beta\right)}$. The corresponding
upper bound for $c(\sigma,\,A^{p}\left(\beta\right),\,H^{\infty})$
will follow by density. 
\begin{proof}
\textbf{Case 1: $\beta\leq0$.} First we prove the upper bound for
$\beta\in(-1,0]$. Let $f\in A^{p}\left(\beta\right)\cap H^{\infty}$
be such that $\Norm{f}{A^{p}\left(\beta\right)}\leq1$. Let 
\[
g(\zeta)=(P_{B}f)(\zeta)=\langle f,\,k_{\zeta}^{B}\rangle.
\]
Applying the Green formula \eqref{gr} to $\phi=k_{\zeta}^{B}$ and
$\psi=f$ we obtain 
\begin{equation}
\overline{g(\zeta)}-\overline{f(0)}k_{\zeta}^{B}(0)=((k_{\zeta}^{B})',\,S^{*}f)=\int_{\mathbb{D}}\big(k_{\zeta}^{B})'(u)\overline{S^{*}f(u)}{\rm d}\mathcal{A}(u).\label{eq:green}
\end{equation}
We first assume that $p>1$ so that its conjugate exponent $q$ (i.e., $\frac{1}{p}+\frac{1}{q}=1$) is finite. Applying the H\"older inequality to $(1-\vert u\vert)^{\beta/p}\abs{S^{*}f(u)}$ and  $(1-\vert u\vert)^{-\beta/p}\Abs{(k_{\zeta}^{B})'(u)}$ with exponents $p$ and $q$, we obtain 
\[
\Abs{((k_{\zeta}^{B})',\,S^{*}f)}\leq\Norm{S^{*}f}{A^{p}\left(\beta\right)}\Norm{(k_{\zeta}^{B})'}{A^{q}(-(q-1)\beta)},
\]
and the estimate for $\|g\|_{\infty}$ follows by a direct application
of \eqref{eq:1st_estimate_AntonLemma} with $\gamma=-(q-1)\beta\in[0,q-1)$
(note that $S^{*}$ is bounded from $A^{p}(\beta)$ onto itself and
$\vert k_{\zeta}^{B}(0)\vert\leq2$).

If $p=1$ then 
\[
\vert((k_{\zeta}^{B})',\,S^{*}f)\vert\lesssim\Norm{f}{A^{1}\left(\beta\right)}\Norm{k_{\zeta}^{B}}{\mathcal{B}_{-\beta}}\lesssim\Norm{f}{A^{1}\left(\beta\right)}\left(\frac{n}{1-r}\right)^{\beta+2}
\]
by \eqref{eq:Bloch_norms} with $l=0$ and $\alpha=-\beta$. \medskip{}
\\
 \textbf{Case 2: $\beta>0$.} Now we prove the upper bound for $\beta>0$.
Applying \eqref{eq:green} and \eqref{gr2} with $l=\big[\frac{\beta}{p}\big]+1$
we get 
\begin{equation}
\big((k_{\zeta}^{B})',\,S^{*}f\big)=\int_{\mathbb{D}}D_{l}\left((k_{\zeta}^{B})'\right)(u)\overline{S^{*}f(u)}(1-\vert u\vert^{2})^{l}{\rm d}\mathcal{A}(u).\label{eq:green_modified}
\end{equation}
Again we first assume that $p>1$ so that its conjugate exponent $q$
is finite. Writing $(1-\vert u\vert^{2})^{l}=(1-\vert u\vert^{2})^{\frac{\beta}{p}+l-\frac{\beta}{p}}$
and applying the H\"older inequality to $(1-\vert u\vert^{2})^{\frac{\beta}{p}}\overline{S^{*}f(u)}$
and $(1-\vert u\vert^{2})^{l-\frac{\beta}{p}}D_{l}\left(\big(k_{\zeta}^{B}\big)'\right)(u)$
we get 
\[
\vert((k_{\zeta}^{B})',\,S^{*}f)\vert\leq\Norm{S^{*}f}{A^{p}\left(\beta\right)}\left(\int_{\mathbb{D}}(1-\vert u\vert)^{q\alpha}\vert D_{l}\left((k_{\zeta}^{B})'\right)(u)\vert^{q}{\rm d}\mathcal{A}(u)\right)^{\frac{1}{q}},
\]
where $\alpha=l-\frac{\beta}{p}$. It remains to apply Corollary \ref{cor:Anton_Cor}
with $\gamma=q\alpha\in[0,q]$: the result follows since 
\[
\Norm{D_{l}\left((k_{\zeta}^{B})'\right)}{A^{q}(q\alpha)}\asymp\Norm{(k_{\zeta}^{B})^{(l+1)}}{A^{q}(q\alpha)}.
\]

If $p=1$ then, by \eqref{eq:Bloch_norms}, 
\[
\vert((k_{\zeta}^{B})',\,S^{*}f)\vert\lesssim\Norm{f}{A^{1}\left(\beta\right)}\Norm{(k_{\zeta}^{B})^{(l)}}{\mathcal{B}_{\alpha}}\lesssim\Norm{f}{A^{1}\left(\beta\right)}\left(\frac{n}{1-r}\right)^{\beta+2}.
\]
\end{proof}

\subsection{The upper bound in \eqref{eq:embedd_bergman}}
\begin{proof}
The upper bound in Theorem \ref{thm:Embedd} follows directly from
the following observation: let $f\in\mathcal{R}_{n}$ and $1/\overline{\lambda}_{1},\dots,1/\overline{\lambda}_{n}$
are the poles of $f$ (repeated according to multiplicities), then
$f\in K_{zB}$ with $\sigma=(\lambda_{1},\dots,\lambda_{n})$. In
particular we have $f=P_{\tilde{B}}f=\left\langle f,\,k_{\zeta}^{\tilde{B}}\right\rangle $,
where $\tilde{B}(z)=zB(z)$. Now we can repeat the above proof for
$\tilde{B}$ instead of $B$. 
\end{proof}
\bigskip{}


\section{\label{sec:Proof-of-the}Proof of the lower bounds }

In this section we estimate from below the interpolation constant
$c(\sigma,\;X,\;H^{\infty})$ for the one-point interpolation sequence
$\sigma_{\lambda,\,n}=(\lambda,\lambda,...,\lambda)\in\mathbb{D}^{n}$:
\[
c(\sigma_{n,\,\lambda},\;X,\;H^{\infty})=\sup\{\Norm{f}{H^{\infty}/b_{\lambda}^{n}H^{\infty}}:\,f\in X\cap H^{\infty},\,\Norm{f}{X}\leq1\},
\]
where $\Norm{f}{H^{\infty}/b_{\lambda}^{n}H^{\infty}}=\mbox{inf}\{\Norm{f+b_{\lambda}^{n}g}{\infty}:\:g\in X\cap H^{\infty}\}$
(recall that $b_{\lambda}(z)=\frac{\lambda-z}{1-\overline{\lambda}z}$).
Since the spaces $X=H^{p},\,A^{p}(\beta)$ and $H^{\infty}$ are rotation
invariant we have $c\left(\sigma_{n,\,\lambda},X,\,H^{\infty}\right)=c\left(\sigma_{n,\,\mu},X,\,H^{\infty}\right)$
for every $\lambda,\,\mu$ with $\vert\lambda\vert=\vert\mu\vert=r$.
Without loss of generality we can thus suppose that $\lambda=-r$.


\subsection{The lower bounds in Theorem \ref{thm:Interp}}

Recall that we need to prove the following estimates: 
\begin{equation}
c\left(\sigma_{n,\,-r},\,H^{p},H^{\infty}\right)\gtrsim\left(\frac{n}{1-r}\right)^{\frac{1}{p}}\label{eq:lower_bd_strength_hardy}
\end{equation}
and 
\begin{equation}
c\left(\sigma_{n,\,-r},\,A^{p}(\beta),\,H^{\infty}\right)\gtrsim\bigg(\frac{n}{1-r}\bigg)^{\frac{2+\beta}{p}}\label{eq:lower_bd_strength_bergman}
\end{equation}
for any $n\geq1$, $r\in[0,\,1)$, $p\in[1,\,+\infty)$ and $\beta>-1$.
\medskip{}

\begin{proof}
For $N\in\mathbb{N}$, we consider the test function 
\begin{equation}
\phi_{n}:=Q_{n}^{N},\label{eq:test_func_Hp-1}
\end{equation}
where 
\[
Q_{n}=\frac{1-r^{2}}{(1+rz)^{2}}\bigg(\sum_{k=0}^{n-1}b_{-r}^{k}(z)\bigg)^{2}=\frac{1-r^{2}}{(1+rz)^{2}}D_{n}^{2}(b_{-r}(z))
\]
and $D_{n}(z)=\sum_{j=0}^{n-1}z^{j}$ is the (analytic part of) the
$n^{\rm{th}}$ Dirichlet kernel. We have 
\[
c\left(\sigma_{n,\,-r},\,X,H^{\infty}\right)\geq\frac{\Norm{\phi_{n}}{H^{\infty}/b_{-r}^{n}H^{\infty}}}{\Norm{\phi_{n}}{X}}.
\]
Thus we need to obtain an upper estimate for $\Norm{\phi_{n}}{X}$
and a lower one for $\Norm{\phi_{n}}{H^{\infty}/b_{-r}^{n}H^{\infty}}.$
\medskip{}
\\
 \textbf{Step 1. Upper estimate for $\Norm{\phi_{n}}{H^{p}}$, $N=1$.}
Note that $Q_{n}=(\sum_{k=0}^{n-1}e_{k})^{2}$, where $e_{k}$ are
the elements of the Malmquist--Walsh basis. Hence, $\Norm{Q_{n}}{H^{1}}=\,n$.
Now we compute $\Norm{Q_{n}}{\infty}.$ Note that $Q_{n}\circ b_{-r}$
is a polynomial of degree $2n-2$ with positive coefficients: indeed,
\[
Q_{n}\circ b_{-r}=\left(\sum_{k=0}^{n-1}z^{k}\frac{(1-r^{2})^{1/2}}{1+rb_{-r}(z)}\right)^{2}=\left(1-r^{2}\right)^{-1}\left(1+(1+r)\sum_{k=1}^{n-1}z^{k}+rz^{n}\right)^{2}.
\]
In particular, 
\begin{equation}
\Norm{Q_{n}}{\infty}=\Norm{Q_{n}\circ b_{-r}}{\infty}=Q_{n}\circ b_{-r}(1)=n^{2}\frac{1+r}{1-r},\label{eq:Q_n_Infty_norm}
\end{equation}
and for any $p\geq1$, $\Norm{Q_{n}}{H^{p}}^{p}\leq\Norm{Q_{n}}{H^{1}}\Norm{Q_{n}}{\infty}^{p-1}$.
Thus 
\begin{equation}
\Norm{Q_{n}}{H^{p}}\leq n^{2-\frac{1}{p}}\left(\frac{1+r}{1-r}\right)^{1-\frac{1}{p}}.\label{bro1}
\end{equation}
\medskip{}
 \textbf{Step 2. Upper estimate for $\Norm{\phi_{n}}{A^{p}(\beta)}$}.
Assume that $\beta\in(l-1,\,l]$ where $l\geq0$ is an integer and
put $N=l+2$. We will prove that 
\begin{equation}
\Norm{\phi_{n}}{A^{p}(\beta)}\lesssim\frac{n^{2N-\frac{\beta+2}{p}}}{(1-r)^{N-\frac{\beta+2}{p}}}.\label{bro2}
\end{equation}
The change of variable $w=b_{-r}(z)$ (equivalently, $z=b_{-r}(w)$)
gives 
\[
\int_{\mathbb{D}}f(b_{-r}(z))\abs{b_{-r}'(z)}^{2}{\rm d}\mathcal{A}(z)=\int_{\mathbb{D}}f(w){\rm d}\mathcal{A}(w)
\]
for any function $f$ summable with respect to $\mathcal{A}$. Then
we have 
\begin{eqnarray*}
\Norm{Q_{n}^{N}}{A^{p}(\beta)}^{p} & = & \frac{1}{(1-r^{2})^{pN-2-\beta}}\int_{\mathbb{D}}\abs{1+rw}^{2pN-4-2\beta}\abs{D_{n}(w)}^{2Np}(1-\abs{w}^{2})^{\beta}{\rm d}\mathcal{A}(w)\\
 & \lesssim & \frac{1}{(1-r)^{pN-2-\beta}}\int_{\mathbb{D}}\abs{D_{n}(w)}^{2Np}(1-\abs{w}^{2})^{\beta}{\rm d}\mathcal{A}(w)
\end{eqnarray*}
since $pN\geq N\geq\beta+2$ and so $2pN-4-2\beta\geq0$. It remains
to see that 
\[
\int_{\mathbb{D}}\abs{D_{n}(w)}^{2Np}(1-\abs{w}^{2})^{\beta}{\rm d}\mathcal{A}(w)\lesssim n^{2pN-2-\beta}.
\]
Indeed, for $p=1$ we have by a very rough estimate 
\[
\Norm{D_{n}^{N}}{A^{2}(\beta)}^{2}=\sum_{k=0}^{(n-1)N}\frac{\widehat{|D_{n}^{N}}(k)|^{2}}{k^{1+\beta}}\lesssim\sum_{k=1}^{(n-1)N}\frac{k^{2N-2}}{k^{1+\beta}}\lesssim n^{2N-2\beta-2},
\]
while for $p\in[1,\,\infty)$, 
\begin{eqnarray*}
\int_{\mathbb{D}}\abs{D_{n}(w)}^{2Np}(1-\abs{w}^{2})^{\beta}{\rm d}\mathcal{A}(w) & \leq & \Norm{D_{n}^{2N}}{\infty}^{p-1}\Norm{D_{n}^{N}}{A^{2}(\beta)}^{2}\\
 & \lesssim & n^{2N(p-1)}n^{2N-2\beta-2}=n^{2pN-2-\beta}.
\end{eqnarray*}
This completes the proof of \eqref{bro2}. \medskip{}
 \\
 \textbf{Step 3. Lower estimate for $\Norm{\phi_{n}}{H^{\infty}/b_{-r}^{n}H^{\infty}}$.}
Put $\Psi_{n}:=\,\phi_{n}\circ b_{-r}$. Clearly, 
\[
\Norm{\phi_{n}}{H^{\infty}/b_{-r}^{n}H^{\infty}}=\Norm{\Psi_{n}}{H^{\infty}/z^{n}H^{\infty}}.
\]
We will show that 
\begin{equation}
\Norm{\Psi_{n}}{H^{\infty}/z^{n}H^{\infty}}\gtrsim\frac{n^{2N}}{(1-r)^{N}}.\label{bab}
\end{equation}

\noindent Denote by $F_{n}$ the $n$-th Fejer kernel, $F_{n}(z)=\frac{1}{2\pi}\sum_{|j|\le n}\Big(1-\frac{|j|}{n}\Big)z^{j}$,
and denote by $*$ the usual convolution operation in $L^{1}(\mathbb{T})$.
Then, for any $g\in L^{\infty}\left(\mathbb{T}\right)$, we have $\Norm{g*F_{n}}{\infty}\le\Norm{g}{\infty}\Norm{F_{n}}{H^{1}}=\Norm{g}{\infty}$.
On the other hand, since $\widehat{g*h}(j)=\hat{g}(j)\hat{h}(j)$
and $\widehat{F}_{n}(j)=0$ for every $j\ge n$, we have 
\[
g*F_{n}=\Psi_{n}*F_{n}
\]
for any $g\in H^{\infty}$ such that $\hat{g}(k)=\widehat{\Psi}_{n}(k)$,
$k=0,1,,\dots,n-1$. Hence, for any such $g$, $\Norm{g}{\infty}\ge\Norm{\Psi_{n}*F_{n}}{\infty}$
and so
\[
\begin{aligned}\Norm{\Psi_{n}}{H^{\infty}/z^{n}H^{\infty}}=\inf\big\{\|g\|_{\infty}:\:g\in H^{\infty},\:\hat{g}(k) & =\hat{\Psi}_{n}(k),\:0\le k\le n-1\big\}\\
 & \geq\Norm{\Psi_{n}*F_{n}}{\infty}\ge(\Psi_{n}*F_{n})(1).
\end{aligned}
\]
Note that the convolution with $F_{n}$ gives us the Ces\`aro mean of
the partial sums of the Fourier series. Denote by $S_{j}$ the $j$-th
partial sum for $\Psi_{n}$ at 1. Recall that 
\[
\Psi_{n}(z)=\frac{1}{{(1-r^{2})^{N}}}\bigg(1+(1+r)\sum_{k=1}^{n-1}z^{k}+rz^{n}\bigg)^{2N}.
\]
Since all Taylor coefficients for $\Psi_{n}$ are positive, we have
\[
\begin{aligned}S_{j}(1) & \ge\frac{1}{{(1-r^{2})^{N}}}\bigg(1+(1+r)\sum_{k=1}^{[(2N)^{-1}j]}z^{k}\bigg)^{2N}\bigg|_{z=1}\\
 & \ge\frac{(1+(1+r)[(2N)^{-1}j])^{2N}}{(1-r^{2})^{N}}\gtrsim\frac{j^{2N}}{(1-r)^{N}}
\end{aligned}
\]
with the constants depending on $N$ only. Hence, 
\[
(\Psi_{n}*F_{n})(1)=\frac{1}{n}\sum_{j=0}^{n-1}S_{j}(1)\gtrsim\frac{n^{2N}}{(1-r)^{N}},
\]
which proves \eqref{bab}. \medskip{}
 \\
 \textbf{Step 4. Completion of the proof.} The estimate \eqref{eq:lower_bd_strength_hardy}
follows from \eqref{bro1} and \eqref{bab} (with $N=1$ and $\phi_{n}=Q_{n}$).
Combining \eqref{bro2} and \eqref{bab} we arrive at the estimate
\eqref{eq:lower_bd_strength_bergman}. 
\end{proof}
\medskip{}

\subsection{The lower bounds in Theorem \ref{thm:Embedd}}
\begin{proof}
We prove the lower bound for $\mathcal{E}_{n,\,r}(H^{\infty},\,A^{p}(\beta))$
in \eqref{eq:embedd_bergman}. We put $N=l+2$, where $l\geq0$ is
the integer such that $\beta\in(l-1,\,l]$, and consider the test
function $\phi_{m}$ defined in \eqref{eq:test_func_Hp-1} with $m=\big[\frac{n}{2N}\big]$
(assuming that $n>2N$). Therefore 
\[
\phi_{m}=\bigg(\sum_{k=0}^{m-1}(1-\vert r\vert^{2})^{1/2}b_{-r}^{k}\left(1+rz\right)^{-1}\bigg)^{2N}\in\mathcal{R}_{n,\,r}.
\]
We know from \eqref{eq:Q_n_Infty_norm} that 
\[
\Norm{Q_{m}^{N}}{\infty}=m^{2N}\left(\frac{1+r}{1-r}\right)^{N},
\]
and it follows from \eqref{bro2} that 
\[
\Norm{Q_{m}^{N}}{A^{p}(\beta)}\lesssim\frac{m^{2N-\frac{\beta+2}{p}}}{(1-r)^{N-\frac{\beta+2}{p}}}\lesssim\frac{m^{-\frac{\beta+2}{p}}}{(1-r)^{-\frac{\beta+2}{p}}}\Norm{Q_{m}^{N}}{\infty}
\]
which completes the proof. 
\end{proof}

\end{document}